\documentclass[11pt,a4paper]{article}
\pdfoutput=1

\usepackage[T1]{fontenc}
\usepackage[utf8]{inputenc}
\usepackage{amsmath,amssymb,amsthm,mathtools}
\usepackage{microtype}
\usepackage[hidelinks]{hyperref}
\usepackage[top=26mm,bottom=27mm,left=30mm,right=30mm,includefoot]{geometry}

\allowdisplaybreaks[1]

\newtheorem{theorem}{Theorem}[section]
\newtheorem{proposition}[theorem]{Proposition}
\newtheorem{lemma}[theorem]{Lemma}

\newcommand{\Z}{\mathbb Z}
\newcommand{\Q}{\mathbb Q}
\newcommand{\Pcal}{\mathcal P}
\newcommand{\card}[1]{\lvert #1\rvert}

\title{Improved Bounds for Distinct Multiples in Intervals}

\author{
Kaizhe Chen\thanks{Department of Mathematics, Princeton University, Princeton, USA. Email: \texttt{kc8556@princeton.edu}.} \qquad
Samuel Korsky\thanks{Independent researcher. Email: \texttt{samkorsky11@gmail.com}.}
}
\date{August 4, 2026}

\begin{document}

\maketitle

\begin{abstract}
For $n\ge1$, let $F(n)$ be the least $H$ such that any $H$ consecutive integers contain $n$ pairwise distinct integers $a_1, a_2, \dots, a_n$ with $k \mid a_k$ for $1\le k\le n$, and define $h_{\mathbb P}(n)$ analogously for the primes at most $n$. We prove
\[
 F(n)\le n^{4/3}\exp\!\left(O\!\left(\frac{\log n}{\log\log n}\right)\right),
 \qquad
 h_{\mathbb P}(n)\ll \frac{n^{4/3}}{(\log n)^{1/3}},
\]
and
\[
 F(n) \ge
 h_{\mathbb P}(n)\ge
 n\exp\!\left(
 \left(\frac{\log 2}{2}-o(1)\right)
 \frac{\log n}{\log\log n}
 \right).
\]
The upper bounds follow from a new estimate for unions of arithmetic progressions. The lower bound adapts a quadratic-residue compression construction of Green and Ruzsa.
\end{abstract}

\section{Introduction}

For any integer $m$ and positive integer $H$, write $I(m,H):=\{m+1,\dots, m+H\}$.
For $n\ge1$, let $F(n)$ be the least positive integer $H$ such that, for every $m\in\Z$, the interval $I(m,H)$ contains distinct integers $a_1,\dots,a_n$ satisfying $i\mid a_i$ for $1\le i\le n$. Define $h_{\mathbb P}(n)$ analogously, with the moduli restricted to the primes at most $n$. Clearly, $F(n)\ge h_{\mathbb P}(n)$.

For an interval $I\subset\Z$ and a finite set $S$ of positive integers, put
\[
 \Gamma_I(S):=\bigcup_{d\in S}(I\cap d\Z)
 =\{x\in I:d\mid x\text{ for some }d\in S\}.
\]
Hall's marriage theorem \cite{Hall1935} says that $I$ contains distinct multiples of all members of a finite set $\mathcal D$ if and only if $\card{\Gamma_I(S)}\ge\card S$ for every $S\subseteq\mathcal D$.
Thus, an upper bound on $F(n)$ (or $h_{\mathbb P}(n)$) requires a uniform lower bound for $\card{\Gamma_I(S)}$, while a lower bound follows from an interval $I$ and a subset $S$ for which $\card{\Gamma_I(S)}<\card S$.

In 1980, Erd\H{o}s and Pomerance \cite{ErdosPomerance1980} proved $F(n)\ll n^{3/2}$ and $h_{\mathbb P}(n)\ll n^{3/2}/\sqrt{\log n}$, and conjectured that $F(n)=n^{1+o(1)}$. Erd\H{o}s \cite{Erdos1992} later offered 1000 rupees for a proof of this conjecture. Erd\H{o}s and Selfridge proved $h_{\mathbb P}(n)>(3-o(1))n$; see \cite[Problem~B32]{Guy}. In 1995, Ruzsa \cite{Ruzsa1995} showed that $h_{\mathbb P}(n)/n\to\infty$. Recently, answering a related problem of Erd\H{o}s and Pomerance \cite{ErdosPomerance1980}, van Doorn \cite{vanDoorn2026} proved $F(n)\gg n\log n/\log\log n$. More recently, Kominers \cite{Kominers2026} proved $F(n)\gg n\log n$ and further asked whether $F(n)\ll n\log n$.

In this note, our first result improves the upper bound on $F(n)$ to exponent $4/3$, up to a subpolynomial factor. The main new input is a local estimate for unions of arithmetic progressions, proved in Section~2.

\begin{theorem}\label{thm:main-upper}
We have
\[
 F(n)\le n^{4/3}\exp\!\left(O\!\left(\frac{\log n}{\log\log n}\right)\right).
\]
\end{theorem}

For prime moduli, the same local estimate gives a logarithmic saving.

\begin{theorem}\label{thm:main-prime-upper}
We have
\[
 h_{\mathbb P}(n)\ll \frac{n^{4/3}}{(\log n)^{1/3}}.
\]
\end{theorem}

Our second result gives the following lower bounds on $F(n)$ and
$h_{\mathbb P}(n)$, which answer Kominers's question whether $F(n)\ll n\log n$ in the negative. The proof adapts the quadratic-residue compression construction of Green and Ruzsa \cite{GreenRuzsa2019}.

\begin{theorem}\label{thm:main-lower}
As $n\to\infty$,
\[
 F(n)\ge h_{\mathbb P}(n)\ge
 n\exp\!\left(
 \left(\frac{\log 2}{2}-o(1)\right)
 \frac{\log n}{\log\log n}
 \right).
\]
\end{theorem}



The rest of the paper is organized as follows. In Section~2, we prove the local progression-union estimate. In Sections~3 and 4, we apply it to prove the upper bounds for $F(n)$ and $h_{\mathbb P}(n)$, respectively. In Section~5, we prove the lower bound. In the appendix, we give an independent upper-bound argument based on the Katz--Tao projection estimate.

For $x\ge2$, let $\pi_{\mathbb P}(x):=\card{\{p\le x:p\text{ prime}\}}$.
We will use the prime number theorem and Mertens' estimate in the following forms
\begin{equation}\label{eq:standard-prime-estimates}
 \pi_{\mathbb P}(x)\sim\frac{x}{\log x},
 \qquad
 \sum_{p\le x}\log p\sim x,
 \qquad
 q_r\sim r\log r,
 \qquad
 \prod_{p\le x}\left(1+\frac1p\right)\ll\log x,
\end{equation}
where $q_r$ denotes the $r$th odd prime; see, for example, \cite{MontgomeryVaughan2007}.

\section{A local progression-union estimate}

The problem of bounding unions of arithmetic progressions with pairwise distinct common differences was studied systematically by Gilboa and Pinchasi~\cite{GilboaPinchasi2014}. In the range relevant to our application, our method yields the estimate below, which gives a stronger exponent.

\begin{lemma}\label{lem:local-cube-root}
Let $\mathcal D$ be a nonempty set of distinct positive integers. For each $d\in\mathcal D$, let
\[
 S_d=\{a_d+jd:0\le j<L\}
\]
be an arithmetic progression of the same even length $L\ge4$. Suppose that all the sets $S_d$ lie in one integer interval $J$ of length $H$. 
Let $U$ be the union of the sets $S_d$ over all $d\in\mathcal D$,
and define
\[
 \kappa:=\max_{0< t\le H}
 \card{\{d\in\mathcal D:d\mid t\}}.
\]
If $|U|\le |\mathcal D|$, then $L^3\le C\kappa^2{|\mathcal D|}$,
where $C$ is an absolute constant.
\end{lemma}

\begin{proof}
Write $L=2\ell$. We split each progression according to the parity of its index so that the midpoint of any two points in one subprogression is an integer point of the original progression. Put
\[
 S_{d,\varepsilon}:=\{a_d+(2j+\varepsilon)d:0\le j<\ell\},
 \qquad \varepsilon\in\{0,1\}.
\]
Let $G$ be the bipartite graph whose left vertices are the $2|\mathcal D|$ sets $S_{d,\varepsilon}$, whose right vertices are the elements of $U$, and where an edge joins $S_{d,\varepsilon}$ and $x\in U$ if and only if $x\in S_{d,\varepsilon}$. Every left vertex has degree $\ell$. For a right vertex $x\in U$, denote by $\deg(x)$ its degree in $G$.

Let $\mathcal T$ be the set of all oriented paths $S_{d,\varepsilon}-u-S_{e,\delta}-w$ of length 3 in $G$.
Choosing the middle point $u$, the ordered pair of distinct neighbors of $u$, and then the final point $w$ gives the exact count
\begin{equation}\label{eq:local-path-identity}
 \card{\mathcal T}
 =(\ell-1)\sum_{u\in U}\deg(u)(\deg(u)-1).
\end{equation}
There are $2{|\mathcal D|}\ell={|\mathcal D|}L$ edges in $G$. Hence, the Cauchy--Schwarz inequality gives
\[
 \sum_{u\in U}\deg(u)(\deg(u)-1)
 \ge \frac{{|\mathcal D|}^2L^2}{|U|}-{|\mathcal D|}L.
\]
Using $|U|\le {|\mathcal D|}$ and $L\ge4$ in \eqref{eq:local-path-identity}, we obtain
\begin{equation}\label{eq:local-path-lower}
 \card{\mathcal T}\gg {|\mathcal D|}L^3.
\end{equation}

We next give an upper bound on $\card{\mathcal T}$. 
Given a path $S_{d,\varepsilon}-u-S_{e,\delta}-w$, put $x:=(u+w)/2$.
The definition of $S_{e,\delta}$ implies $x\in S_e\subset U$. Since $u\in S_{d,\varepsilon}$, we have $u\equiv a_d\pmod d$ and hence
\begin{equation}\label{eq:local-reflection-congruence}
 w\equiv2x-a_d\pmod d.
\end{equation}
Map the path to the tuple $(d,x,w)$. For a fixed image tuple, the point $u=2x-w$ is determined. 
Moreover, $d$ and $u$ determine $\varepsilon$.
Any possible middle difference $e$ must divide the nonzero integer $w-u$. Since $u,w\in J$, we have $0<\lvert w-u\rvert\le H$, so there are at most $\kappa$ choices for $e$. For any such $e$, $e$ and $u$ determine $\delta$. Thus, the map is at most $\kappa$-to-one, and \eqref{eq:local-reflection-congruence} gives
\begin{equation}\label{eq:local-path-upper-one}
 \card{\mathcal T}
 \le \kappa\sum_{d\in\mathcal D}\sum_{c\bmod d}
 m_d(c)m_d(2c-a_d),
\end{equation}
where $m_d(c):=\card{\{z\in U:z\equiv c\pmod d\}}$ for any $d\in\mathcal D$ and $c\in\Z/d\Z$.

For odd $d$, the map $c\mapsto2c-a_d$ is a permutation modulo $d$. For even $d$, every value has either zero or two preimages. Thus, the Cauchy--Schwarz inequality gives
\begin{align}\label{eq:local-residue-cs}\notag
 \sum_{c\bmod d}m_d(c)m_d(2c-a_d)&\le\Bigl(\sum_{c\bmod d} m_d(c)^2 \Bigr)^{\!1/2} \Bigl(\sum_{c\bmod d} m_d(2c-a_d)^2\Bigr)^{\!1/2}\\ &\le \sqrt2 \sum_{c\bmod d}m_d(c)^2.
\end{align}
Observe that
\begin{align*}
 \sum_{d\in\mathcal D}\sum_{c\bmod d}m_d(c)^2
 &=\sum_{d\in \mathcal D}\left| \{(u,v):u\in U,v\in U, u\equiv v\! \pmod{d}\} \right|\\
 &=\sum_{u,v\in U}\left|\{d\in \mathcal D:u\equiv v\! \pmod{d}\}\right|\\
 &={|\mathcal D|}|U|+\sum_{\substack{u,v\in U\\u\ne v}}
 \card{\{d\in\mathcal D:u\equiv v\! \pmod{d}\}}.
\end{align*}
Every nonzero difference $u-v$ has absolute value at most $H$, so the last cardinality is at most $\kappa$. 
Since $S_d\subset J$ has diameter $(L-1)d$ and $J$ has
length $H$, we have $d\le H$ for every $d\in\mathcal D$.
Consequently $\kappa\ge1$.
Since $|U|\le {|\mathcal D|}$, we derive
\begin{equation}\label{eq:local-energy-upper}
 \sum_{d\in\mathcal D}\sum_{c\bmod d}m_d(c)^2
 \le {|\mathcal D|}|U|+\kappa |U|(|U|-1)
 \le2\kappa {|\mathcal D|}^2.
\end{equation}
Combining \eqref{eq:local-path-upper-one}--\eqref{eq:local-energy-upper} gives
$\card{\mathcal T}\ll\kappa^2{|\mathcal D|}^2$, which together with \eqref{eq:local-path-lower} proves the lemma.
\end{proof}


\section{The all-integer upper bound}

We first record the elementary divisor estimate used to control $\kappa$. For any integer $t>0$, let $\tau(t)$ denote its number of positive divisors. For any integer $n>0$, let $\Delta_n$ be the maximum $\tau(t)$ over all $1\le t\le n^2$.

\begin{lemma}\label{lem:divisor-bound}
As $n\to\infty$,
\[
 \Delta_n \le \exp\!\left(O\!\left(\frac{\log n}{\log\log n}\right)\right).
\]
\end{lemma}

\begin{proof}
Fix $t\le {n^2}$ and write $t=\prod_p p^{a_p}$. For $p\le \sqrt{\log {n}}$, we have $a_p\le 2\log_2 {n}$, and hence
\begin{equation}\label{small}
\log\prod_{p\le \sqrt{\log {n}}}(a_p+1)
 \le \sqrt{\log {n}}\log\!\left(1+2\log_2 {n}\right)
 =o\!\left(\frac{\log {n}}{\log\log {n}}\right).
\end{equation}
For $p>\sqrt{\log {n}}$, the inequality $\sum_{p}a_p\log p\le 2\log {n}$ gives
\[
 \sum_{p>\sqrt{\log {n}}}a_p
 \le\frac{2\log {n}}{\log \sqrt{\log {n}}}
 =\frac{4\log {n}}{\log\log {n}}.
\]
Since $a+1\le2^a$ for every positive integer $a$, we have
\begin{equation}\label{big}
 \log\prod_{p>\sqrt{\log {n}}}(a_p+1)
 \le(\log2)\sum_{p>\sqrt{\log {n}}}a_p
 =O\!\left(\frac{\log {n}}{\log\log {n}}\right).
\end{equation}
Summing \eqref{small} and \eqref{big} yields the desired estimate.
\end{proof}

\begin{proof}[Proof of Theorem~\ref{thm:main-upper}]
Let $C\ge 1$ be given by Lemma~\ref{lem:local-cube-root}.
Choose an absolute constant $A\ge2$ such that $8A^3>C$.
Fix an integer $n$, put $K:=\lceil\log_2 n\rceil$, and, for $0\le k\le K$, define
\[
 \mathcal D_k:=\{d\in\{1,\dots,n\}:2^{k-1}<d\le 2^k\}.
\]
These classes partition $\{1,\dots,n\}$. For each $0\le k\le K$, put
\begin{equation}\label{eq:integer-block-parameters}
 L_k:=2\left\lceil A\left(1+\Delta_n^{2/3}{|\mathcal D_k|}^{1/3}\right)\right\rceil,
 \qquad H_k:=2^kL_k,\qquad {\rm and}\qquad H:=\sum_{0\le k\le K}H_k.
\end{equation}

Fix $m\in\Z$ and partition $I(m,H)$ into pairwise disjoint intervals $J_k$ of lengths $H_k$ for $0\le k\le K$. Since every $d\in\mathcal D_k$ satisfies $d\le 2^k$, the block $J_k$ contains at least $L_k$ multiples of $d$. Choose $L_k$ consecutive multiples of each $d\in\mathcal D_k$ inside $J_k$, and denote their set by $S_d$.

We claim that, for any subset $S\subseteq \mathcal D_k$, we have $\left|\bigcup_{d\in S}S_d\right|\ge |S|$. Suppose otherwise. Then there is a nonempty set $S\subseteq\mathcal D_k$ with ${\card U}<{\card S}$, where $U$ is the union of $S_d$ over all $d\in S$.
By Lemma~\ref{lem:divisor-bound}, \eqref{eq:integer-block-parameters}, and the bounds $2^k\le2n$ and ${|\mathcal D_k|}\le 2^k$,
\begin{align*}
 H_k
 \ll 2^k\left(1+\Delta_n^{2/3}{|\mathcal D_k|}^{1/3}\right)
 \le 2n\left(1+\Delta_n^{2/3}(2n)^{1/3}\right)
 \ll n^{4/3+o(1)}<n^2
\end{align*}
for all sufficiently large $n$, uniformly in $k$. Hence every nonzero difference of two points of $J_k$ has absolute value at most $n^2$. The parameter $\kappa$ in Lemma~\ref{lem:local-cube-root}, applied to the subfamily indexed by $S$, is therefore at most $\Delta_n$. That lemma gives $L_k^3\le C\Delta_n^2{\card S}\le C\Delta_n^2{|\mathcal D_k|}$.
On the other hand, \eqref{eq:integer-block-parameters} gives $L_k^3\ge 8A^3\Delta_n^2{|\mathcal D_k|}>C\Delta_n^2{|\mathcal D_k|}$, a contradiction. Thus, Hall's condition holds for every $k$.

By Hall's marriage theorem, the sets $\{S_d:d\in\mathcal D_k\}$ admit a system of distinct representatives in $J_k$. Since the blocks $J_k$ are pairwise disjoint, these systems combine to give distinct integers $a_1,\dots, a_n\in I(m,H)$ such that $d\mid a_d$ for every $1\le d\le n$. As $m$ was arbitrary, $F(n)\le H$. Finally,
\begin{align*}
 H\ll \sum_{k=0}^K 2^k +\Delta_n^{2/3}\sum_{k=0}^K 2^k{|\mathcal D_k|}^{1/3} \ll n +\Delta_n^{2/3}\sum_{k=0}^K {2}^{4k/3} \ll \Delta_n^{2/3}n^{4/3}.
\end{align*}
The desired bound then follows from Lemma~\ref{lem:divisor-bound}.
\end{proof}

\section{The prime upper bound}

In this section, we apply Lemma~\ref{lem:local-cube-root} to prove Theorem~\ref{thm:main-prime-upper}. 

\begin{proof}[Proof of Theorem~\ref{thm:main-prime-upper}]
Let $C\ge 1$ be given by Lemma~\ref{lem:local-cube-root}, and choose an absolute constant $A\ge 2$ with $8A^3>C$. Choose an absolute constant $X_\ast$ such that, for any $X\ge X_\ast$,
\begin{equation}\label{eq:prime-block-short}
 2(A+1)X^{4/3}+1<\frac{X^2}{4}.
\end{equation}
Enlarging the implicit constant if necessary, we may assume $n\ge 2X_\ast$. Put $X_j=2^{-j}n$, and let $j_0$ be the largest integer such that $X_{j_0}\ge2X_\ast$. For $0\le j\le j_0$, set
\[
 \Pcal_j:=\{p\text{ prime}:p>2X_\ast,\ X_j/2<p\le X_j\}.
\]
The sets $\Pcal_j$ partition the primes in $(2X_\ast,n]$. For each nonempty $\Pcal_j$, put
\[
 L_j:=2\left\lceil A{\card{\Pcal_j}}^{1/3}\right\rceil
 \qquad {\rm and}\qquad
 H_j:=\left\lceil X_jL_j\right\rceil.
\]
Also put
\[
 H_\ast:=\sum_{p\le2X_\ast}p \qquad {\rm and}\qquad
 H:=H_\ast+\sum_{\Pcal_j\ne\varnothing}H_j.
\]

Let $I$ be an interval of length $H$. Partition its first $H_\ast$ elements into consecutive blocks of lengths $p$, one for each prime $p\le2X_\ast$. Each such block contains a multiple of the corresponding prime, so these disjoint blocks provide distinct multiples of all primes at most $2X_\ast$. Partition the remaining elements of $I$ into consecutive intervals $I_j$ of lengths $H_j$ for each $j$ with $\Pcal_j\ne\varnothing$.

Fix $j$ with ${\card{\Pcal_j}}>0$, and let $\varnothing\ne S\subseteq\Pcal_j$. Since every $p\in S$ satisfies $p\le X_j$ and $H_j\ge X_jL_j$, the interval $I_j$ contains at least $L_j$ consecutive multiples of $p$. Choose $L_j$ of them for each $p\in S$, and let $U$ be their union. Since ${\card{\Pcal_j}}\ge1$,
\[
 L_j\le2(A+1){\card{\Pcal_j}}^{1/3}\le2(A+1)X_j^{1/3},
\]
and hence \eqref{eq:prime-block-short} gives
\[
 \card{I_j}=H_j
 \le2(A+1)X_j^{4/3}+1
 <\frac{X_j^2}{4}<pq
\]
for any distinct $p,q\in S$. Thus, every integer $t$ with $0<t\le H_j$ is divisible by at most one prime in $S$, so the parameter $\kappa$ in Lemma~\ref{lem:local-cube-root} equals $1$.
If $\card U\le\card S$, that lemma would give $L_j^3\le C\card S\le C{\card{\Pcal_j}}$, whereas the definition of $L_j$ gives $L_j^3\ge8A^3{\card{\Pcal_j}}>C{\card{\Pcal_j}}$.
So, $\card U>\card S$. 
Since $U\subseteq \Gamma_{I_j}(S)$, we have $\card{\Gamma_{I_j}(S)}>\card S$ for every nonempty $S\subseteq\Pcal_j$. Hence, by Hall's marriage theorem, there exist distinct integers in $I_j$, one divisible by each prime in $\Pcal_j$. Since the intervals $I_j$ are pairwise disjoint, these choices, together with those for the primes at most $2X_\ast$, give distinct multiples of every prime at most $n$ in $I$. So, $h_{\mathbb P}(n)\le H$.

Note that $j_0=O(\log n)$, $\sum_{j\ge0}X_j\le 2n$, and ${\card{\Pcal_j}}\le\pi_{\mathbb P}(n)$. We derive
\begin{align*}
 H\le H_\ast+O(\log n)+\sum_{\Pcal_j\ne\varnothing}X_jL_j
 \ll n+\sum_{\Pcal_j\ne\varnothing}X_j{\card{\Pcal_j}}^{1/3}
 \ll n\pi_{\mathbb P}(n)^{1/3}.
\end{align*}
The desired result then follows from \eqref{eq:standard-prime-estimates}.
\end{proof}

\section{The prime lower bound}

We begin with the quadratic-residue compression used in the construction of Green and Ruzsa \cite[Section~5]{GreenRuzsa2019}.

\begin{lemma}\label{lem:compression}
Let $q_1,\dots,q_r$ be distinct odd primes, and put
\[
 Q=\prod_{i=1}^r q_i,
 \qquad
 \delta=\frac{1}{2^{r}}\prod_{i=1}^r\left(1+\frac1{q_i}\right).
\]
For every integer $j$, the set $\{d^2+jd \pmod Q:d\in\Z \}$ has cardinality at most $Q\delta$.
\end{lemma}

\begin{proof}
Modulo $q_i$,
\[
 d^2+jd=\left(d+\frac j2\right)^2-\frac{j^2}{4},
\]
where division by $2$ is taken modulo $q_i$. Since there are $(q_i+1)/2$ quadratic residues modulo $q_i$, including $0$, the expression $d^2+jd$ takes at most $(q_i+1)/2$ values modulo $q_i$. The Chinese remainder theorem therefore gives at most
\[
 \prod_{i=1}^r\frac{q_i+1}{2}=Q\delta
\]
residue classes modulo $Q$.
\end{proof}

We shall also use the following averaging consequence of the prime number theorem.

\begin{lemma}\label{lem:prime-rich-block}
Let $D=D(n)$ be a positive integer with $D=o(n)$. For all sufficiently large $n$, there is an integer $v$ such that $[v,v+D)\subset[n/2,n)$
and $[v,v+D)$ contains at least $D/(2\log n)$ primes.
\end{lemma}

\begin{proof}
Let $v_0 :=\lceil n/2\rceil$, $R:=\left\lfloor (n-v_0)/D \right\rfloor$, and $x:=v_0+RD$.
For large $n$, $R\ge1$, and the $R$ intervals
\[
 [v_0+jD,v_0+(j+1)D),
 \qquad 0\le j<R,
\]
are disjoint and contained in $[n/2,n)$. Moreover, $n-D<x\le n$. Since $D=o(n)$, the prime number theorem gives
\[
 \pi_{\mathbb P}(x-1)-\pi_{\mathbb P}(v_0-1)
 =\left(\frac12+o(1)\right)\frac{n}{\log n}.
\]
As $R\le n/(2D)$, at least one of the $R$ intervals contains
\[
 \frac{1}{R}\left(\frac12+o(1)\right)\frac{n}{\log n}
\ge \left(1+o(1)\right)\frac{D}{\log n}
\ge \frac{D}{2\log n}
\]
primes.
\end{proof}

\begin{proof}[Proof of Theorem~\ref{thm:main-lower}]
Fix $\eta\in(0,1)$. All $o(1)$-terms in this proof are taken with $\eta$ fixed. Write
\[
 r:=\left\lfloor\frac{(1-\eta){\log n}}{{\log\log n}}\right\rfloor.
\]
Let $3=q_1<\cdots<q_r$ be the first $r$ odd primes, and define $Q$ and $\delta$ as in Lemma~\ref{lem:compression}. It follows from \eqref{eq:standard-prime-estimates} that
\begin{align}
 \log Q
 &=-\log2+\sum_{p\le q_r}\log p
 =(1-\eta+o(1)){\log n},\label{eq:Q-size}\\
 \log\delta
 &=-r\log2+O(\log\log q_r)
 =-(\log2+o(1))r.\label{eq:delta-size}
\end{align}

We choose
\[
 K:=\left\lfloor\frac{1}{4\sqrt{\delta\ {\log n}}}\right\rfloor
 \qquad {\rm and}\qquad
 D:=\left\lfloor\frac{n}{4K}\right\rfloor.
\]
Equation \eqref{eq:delta-size} and the fact that $K\to\infty$ give
\begin{equation}\label{eq:K-size}
 \log K=
 \left(\frac{(1-\eta)\log2}{2}+o(1)\right)
 \frac{{\log n}}{{\log\log n}}.
\end{equation}
In particular, $K=n^{o(1)}$, $D=n^{1-o(1)}$, and $D=o(n)$. Moreover, \eqref{eq:Q-size} gives $Q=o(D)$.

Apply Lemma~\ref{lem:prime-rich-block} to this value of $D$. There exist an integer $v$ and a set $\Pcal$ of primes in $[v,v+D)\subset[n/2,n)$ such that
\begin{equation}\label{eq:P-size}
 \card\Pcal\ge\frac{D}{2{\log n}}.
\end{equation}
The largest prime factor of $Q$ is $q_r=O({\log n})$, whereas every $p\in\Pcal$ is at least $n/2$, so $(p,Q)=1$ for sufficiently large $n$. For each $p\in\Pcal$, let $b_p$ be the representative of $p^2\pmod Q$ with $1\le b_p<Q$. For each $0\le j<K$, set $T_j:=\{b_p+jp:p\in\Pcal\}$, and let $T:=\bigcup_{j=0}^{K-1}T_j$.

For each $j$, Lemma~\ref{lem:compression} shows that $T_j$ occupies at most $Q\delta$ residue classes modulo $Q$. Since the primes in $\Pcal$ lie in an interval of length $D$, the diameter of $T_j$ is less than $Q+jD$. A residue class modulo $Q$ contributes at most $1+(Q+jD)/Q$ elements to such an interval, and therefore $\card{T_j}\le(2Q+jD)\delta$.
Summing over $0\le j<K$ and using $Q\le D$ for large $n$,
\begin{equation}\label{eq:T-small}
 \card T \le 2KQ\delta+\frac{K(K-1)}2D\delta \le 3K^2D\delta
 \le\frac{3D}{16{\log n}} <\card\Pcal,
\end{equation}
where the last inequality follows from \eqref{eq:P-size}.

By the Chinese remainder theorem, there is an integer $m$ satisfying $m\equiv-b_p\pmod p$ for every $p\in\Pcal$.
Thus, $p\mid m+b_p+jp$ for every $p\in\Pcal$ and every integer $j$. Put $H:=Kv$. Since $KD\le n/4$, $Q=o(D)$, and $v\ge n/2$, we have
\begin{equation}\label{QKD}
 Q+KD\le Q+\frac n4=o(n)+\frac n4<\frac n2\le v
\end{equation}
for all sufficiently large $n$. We claim that, for each $p\in\Pcal$, the multiples of $p$ in $I(m,H)$ are exactly $\{m+b_p+jp: 0\le j<K\}$.
Indeed, $Q<n/2\le p$ gives $b_p-p<0<b_p$, while
\[
 b_p+(K-1)p
 <Q+(K-1)(v+D)
 <Kv=H
 <b_p+Kp,
\]
where we have used \eqref{QKD} and the assumption $\Pcal\subset [v,v+D)$. Consequently,
\[
 I(m,H)\cap\bigcup_{p\in\Pcal}p\Z=\{m+t:t\in T\}.
\]
By \eqref{eq:T-small}, this set has cardinality smaller than $\card\Pcal$. Hall's theorem therefore rules out distinct multiples for all the primes in $\Pcal$, and hence $h_{\mathbb P}(n)>H$.

Finally, $v\ge n/2$, so \eqref{eq:K-size} gives
\[
 h_{\mathbb P}(n)>H\ge\frac{nK}{2}
 =n\exp\!\left(
 \left(\frac{(1-\eta)\log2}{2}+o(1)\right)
 \frac{\log n}{\log\log n}
 \right).
\]
Since $\eta>0$ was arbitrary, the desired lower bound for $h_{\mathbb P}(n)$ follows. 
\end{proof}

\appendix
\setcounter{section}{1}
\section*{Appendix: A Katz--Tao projection argument}
\addcontentsline{toc}{section}{Appendix: A Katz--Tao projection argument}
\setcounter{theorem}{0}

In this section, we show that the Katz--Tao projection estimate directly leads to an upper bound for $F(n)$. Although this bound is weaker than Theorem~\ref{thm:main-upper}, it already improves upon the previous best bound $F(n)\ll n^{3/2}$ of Erd\H{o}s--Pomerance.

\begin{theorem}\label{thm:katz-tao-upper}
Let $\beta\in(1,2)$ be the root of $2\beta^3-8\beta^2+8\beta-1=0$.
Then,
\[
 F(n)\le n^{\beta+o(1)}.
\]
\end{theorem}

We first state the projection estimate in the rational form needed here. For $r\in\Q$, define $\pi_r(x,y):=x+ry$ and $\pi_\infty(x,y):=y$.
Katz and Tao formulate the statement $\mathrm{SD}(\alpha)$ in terms of finite sets on which $\pi_{-1}$ is injective and prove it for the exponent below \cite[Definition~3.1 and Section~4]{KatzTao2002}. Green and Ruzsa record the rational-projection formulation and the same known exponent \cite[Conjecture~1.3 and the paragraph following Theorem~1.8]{GreenRuzsa2019}. Selecting one representative from each $\pi_{-1}$-fibre gives the following form.

\begin{proposition}[\cite{KatzTao2002}]\label{prop:kt-rational}
Let $\alpha\in(1,2)$ be the root of $\alpha^3-4\alpha+2=0$.
For every $\eta>0$, there exist a finite nonempty set $R_\eta\subset\Q\cup\{\infty\}$ with $-1\notin R_\eta$ and a constant $C_\eta\ge1$ such that, for every finite set $G\subset\Z^2$,
\[
 \card{\pi_{-1}(G)}\le C_\eta
 \left(\max_{r\in R_\eta}\card{\pi_r(G)}\right)^{\alpha+\eta}.
\]
\end{proposition}

The following lemma converts the projection estimate into a lower bound for unions of arithmetic progressions.

\begin{lemma}\label{lem:progression-union}
For every $\eta>0$, there are $k_\eta\ge1$ and $c_\eta>0$ such that the following holds. If a finite set $A\subset\Z$ contains a $k_\eta$-term arithmetic progression of common difference $d$ for every $d$ in a finite set $\mathcal D$, then
\[
 \card A\ge c_\eta\card{\mathcal D}^{1/(\alpha+\eta)}.
\]
\end{lemma}

\begin{proof}
Let $R_\eta$ and $C_\eta$ be supplied by Proposition~\ref{prop:kt-rational}. For finite $r\in R_\eta$, put $s_r=(1+r)^{-1}$; if $\infty\in R_\eta$, put $s_\infty=0$. Choose a positive integer $M$ such that $Ms_r\in\Z$ for every finite $r\in R_\eta$, and then choose $J\in\Z_{\ge0}$ so that $j_r:= J+Ms_r \ge 0$ for each $r\in R_\eta$.
Set
\[
 k_\eta:=1+\max_{r\in R_\eta}j_r
 \qquad {\rm and}\qquad
 c_\eta:=C_\eta^{-1/(\alpha+\eta)}.
\]

By assumption, for each $d\in\mathcal D$, there exists an integer $a_d$ such that $a_d+jd\in A$ for all $0\le j<k_\eta$.
Put $G :=\{(a_d,d):d\in\mathcal D\}$. Define
\[
 \Phi(x,y):=\bigl(x+(J+M)y,\,x+Jy\bigr).
\]
Then $\pi_{-1}(\Phi(a_d,d))=Md$, and hence $\card{\pi_{-1}(\Phi(G))} =\card{\mathcal D}$. For finite $r\in R_\eta$,
\[
 \pi_r(\Phi(a_d,d))
 =(1+r)(a_d+j_rd),
\]
whereas, for $r=\infty$,
\[
 \pi_\infty(\Phi(a_d,d))=a_d+Jd=a_d+j_\infty d.
\]
Since $0\le j_r<k_\eta$, we have $|\pi_{r}(\Phi(G))|\le \card A$ for each $r\in R_\eta$. Proposition~\ref{prop:kt-rational} therefore gives $\card{\mathcal D}\le C_\eta\card A^{\alpha+\eta}$, which is the desired estimate.
\end{proof}

\begin{proof}[Proof of Theorem~\ref{thm:katz-tao-upper}]
Fix $\eta>0$ and put $\theta :=\alpha+\eta$. Let $k:=k_\eta$ and $c:=c_\eta$ be given by Lemma~\ref{lem:progression-union}. Set
\[
 B=\left\lceil c^{-1}n^{1-1/\theta}\right\rceil\qquad {\rm and}\qquad
 H=Bkn.
\]
Fix $m\in\Z$ and partition $I(m,H)$ into $B$ consecutive blocks $I_1,\dots,I_B$ of length $kn$. Let $S$ be a nonempty subset of $\{1,\dots,n \}$. Note that each block $I_b$ contains $k$ consecutive multiples of $d$ for each $d\in S$, since $d\le n$. Thus, Lemma~\ref{lem:progression-union} gives $\card{\Gamma_{I_b}(S)}\ge c\card S^{1/\theta}$. It follows that
\[
 \card{\Gamma_{I(m,H)}(S)}
 \ge\sum_{b=1}^B\card{\Gamma_{I_b}(S)}
 \ge Bc\card S^{1/\theta}
 \ge n^{1-1/\theta}\card S^{1/\theta}
 \ge\card S.
\]
Hall's theorem implies $F(n)\le H\ll_\eta n^{2-1/\theta}$. Letting $\eta\to 0$ gives the desired statement.
\end{proof}

\section*{Statement on AI}

The authors used ChatGPT-5.6 Sol as an exploratory and proof-auditing tool during the development and refinement of the arguments. The authors have carefully checked all mathematical proofs and take responsibility for the manuscript.


\begin{thebibliography}{99}
\bibitem{Erdos1992}
P. Erd\H{o}s, \newblock Some of my forgotten problems in number theory, \newblock \emph{Hardy-Ramanujan J.} \textbf{15} (1992), 34--50.

\bibitem{ErdosPomerance1980}
P.~Erd\H{o}s and C.~Pomerance,
\newblock Matching the natural numbers up to $n$ with distinct multiples in another interval,
\newblock \emph{Indagationes Mathematicae} \textbf{42} (1980), no.~2, 147--161.

\bibitem{GilboaPinchasi2014} S.~Gilboa and R.~Pinchasi, \newblock On the union of arithmetic progressions, \newblock \emph{SIAM J. Discrete Math.} \textbf{28} (2014), no.~3, 1062--1073.

\bibitem{GreenRuzsa2019}
B.~Green and I.~Z. Ruzsa,
\newblock On the arithmetic Kakeya conjecture of Katz and Tao,
\newblock \emph{Periodica Mathematica Hungarica} \textbf{78} (2019), no.~2, 135--151.

\bibitem{Guy}
R.~K. Guy,
\newblock \emph{Unsolved Problems in Number Theory},
\newblock Springer, New York, 2004.

\bibitem{Hall1935}
P.~Hall,
\newblock On representatives of subsets,
\newblock \emph{Journal of the London Mathematical Society} s1-\textbf{10} (1935), no.~1, 26--30.

\bibitem{KatzTao2002}
N.~H. Katz and T.~Tao,
\newblock New bounds for Kakeya problems,
\newblock \emph{Journal d'Analyse Math\'ematique} \textbf{87} (2002), 231--263.

\bibitem{Kominers2026}
S.~D. Kominers,
\newblock Long intervals without distinct multiples of the first $n$ positive integers,
\newblock arXiv:2607.10431, 2026.

\bibitem{MontgomeryVaughan2007}
H.~L. Montgomery and R.~C. Vaughan,
\newblock \emph{Multiplicative Number Theory I: Classical Theory},
\newblock Cambridge Studies in Advanced Mathematics, vol.~97, Cambridge University Press, Cambridge, 2007.

\bibitem{Ruzsa1995}
I.~Z. Ruzsa,
\newblock Few multiples of many primes,
\newblock \emph{Studia Scientiarum Mathematicarum Hungarica} \textbf{30} (1995), no.~1--2, 123--125.

\bibitem{vanDoorn2026}
W.~van Doorn,
\newblock On the length of an interval that contains distinct multiples of the first $n$ positive integers,
\newblock \emph{Integers} \textbf{26} (2026), Article No.~A7.

\end{thebibliography}
\end{document}